\documentclass[a4paper]{amsart}

\def\gam{\gamma }

\newcommand{\set}[1]{\{#1\}}%set

\newcommand{\remove}[1]{ }

\newtheorem{theorem}{Theorem}[section]
\newtheorem{proposition}[theorem]{Proposition}
\newtheorem{lemma}[theorem]{Lemma}

\theoremstyle{definition}
\newtheorem*{definition}{Definition}

\theoremstyle{remark}
\newtheorem*{remark}{Remark}
\newtheorem*{remarks}{Remarks}

\numberwithin{equation}{section}

\begin{document}
\title[Expansions in non-integer bases]{Greedy and quasi-greedy expansions in non-integer bases}
\author{Claudio Baiocchi}
\address{Via regina Margherita 43, 00030 Gavignano, Italy}
\email{claudio.baiocchi-p314@poste.it}
\author{Vilmos Komornik}
\address{D\'epartement de math\'ematique\\
         Universit\'e Louis Pasteur\\
         7, rue Ren\'e Descartes\\
         67084 Strasbourg Cedex, France}
\email{komornik@math.u-strasbg.fr}
%\thanks{}
%\author{}
%\address{}
%\curraddr{}
%\email{}
%\subjclass{}
%\keywords{}
\date{\today}
%\thanks
%\dedicatory

\begin{abstract}
We generalize several theorems of R\'enyi, Parry, Dar\'oczy and K\'atai by characterizing the greedy and quasi-greedy expansions in non-integer bases.
\end{abstract}

\maketitle

\section{Introduction}\label{s1}

Fix a positive integer $M$. By a \emph{sequence} we mean a sequence $(c_i)=c_1c_2\ldots$ satisfying $c_i\in\set{0,1,\ldots,M}$ for each $i$. A sequence is called {\em finite} if it contains only finitely many nonzero terms; otherwise it is called {\em infinite}.

Given a real number $q>1$ and a nonnegative real number $x$, by an \emph{expansion} of  $x$ we mean a sequence $(c_i)$ satisfying \begin{equation*}
\frac{c_1}{q}+ \frac{c_2}{q^2}+\cdots =x.
\end{equation*}
This can only happen if $x\in [0,M/(q-1)]$ because
\begin{equation*}
0\le \frac{c_1}{q}+ \frac{c_2}{q^2}+\cdots\le \frac{M}{q}+ \frac{M}{q^2}+\cdots =\frac{M}{q-1}.
\end{equation*}

If $q\le M+1$, then the converse statement also holds: every $x\in [0,M/(q-1)]$ has at least one expansion. More precisely, we will show that 
every $x\in [0,M/(q-1)]$ has a lexicographically largest expansion and every $x\in (0,M/(q-1)]$ has a lexicographically largest \emph{infinite} expansion. The \emph{lexicographical} order between sequences is defined in the usual way: we write
\begin{equation*}
a<b,\quad (a_i)<(b_i)\quad \text{or}\quad a_1a_2\ldots <b_1b_2\ldots
\end{equation*}
if there exists an index $n\ge 1$ such that
\begin{equation*}
a_1\ldots a_{n-1}=b_1\ldots b_{n-1}\quad \text{and}\quad a_n<b_n.
\end{equation*}
Furthermore, we write
\begin{equation*}
a\le b,\quad (a_i)\le (b_i)\quad \text{or}\quad a_1a_2\ldots \le b_1b_2\ldots
\end{equation*}
if we also allow the equality of the two sequences.

We will give two simple algorithms for the construction of these special expansions, we will characterize them algebraically, and we will compare them. Finally, we also study the limiting case $M=\infty$ when every sequence of nonnegative integers is admitted as a possible expansion.

The results of this note extend various former theorems due to R\'enyi \cite{Ren1957}, Parry \cite{Par1960}, Dar\'oczy and K\'atai \cite{DarKat1995} (see also \cite{KomLor116}) and they have been already applied in the study of unique expansions \cite{KomLor125}, \cite{DevKom2006} and of expansions with deleted digits by Pedicini \cite{Ped}.

\section{Quasi-greedy expansions}\label{s2}

Fix a positive integer $M$ and a real number $q>1$. If $(c_i)=c_1c_2\ldots$ is an infinite expansion of $x$, then
\begin{equation*}
0< \frac{c_1}{q}+ \frac{c_2}{q^2}+\cdots\le \frac{M}{q}+ \frac{M}{q^2}+\cdots =\frac{M}{q-1},
\end{equation*}
so that $x\in (0,M/(q-1)]$. 

In order to prove a converse statement, let us introduce for each $x>0$ the
lexicographically largest infinite sequence $(a_i)$ satisfying
\begin{equation}\label{21}
\frac{a_1}{q}+ \frac{a_2}{q^2}+\cdots \le x.
\end{equation}
This is equivalent to the following recursive definition: if $a_k$ has already
been defined for all $k<n$ (no assumption if $n=1$), then let $a_n$ be the
largest integer satisfying the inequalities
\begin{equation*}
a_n\le M
\quad\text{and}\quad
\frac{a_1}{q}+\cdots + \frac{a_n}{q^n}< x.
\end{equation*}
Since $x>0$, the definition is correct. We have the following

\begin{proposition}\label{p21}
If $M\ge q-1$ and $0< x\le M/(q-1)$, then $(a_i)$ is an infinite expansion of $x$.
\end{proposition}

\begin{proof}
In view of \eqref{21} it suffices to establish the converse inequality
\begin{equation}\label{22}
\frac{a_1}{q}+ \frac{a_2}{q^2}+\cdots \ge x.
\end{equation}

By definition we have
\begin{equation*}
x-\frac{1}{q^n}\le \frac{a_1}{q}+\cdots + \frac{a_n}{q^n}
\quad\text{whenever}\quad
a_n<M.
\end{equation*}
If there are infinitely many such indices then letting $n\to\infty$ hence
\eqref{22} follows.

If $a_i=M$ for all $i$, then 
\begin{equation*}
\frac{a_1}{q}+ \frac{a_2}{q^2}+\cdots = \frac{M}{q-1}\ge x
\end{equation*}
by our assumption on $x$.

Finally, if there exists a last digit $a_n<M$, then 
\begin{align*}
\frac{a_1}{q}+ \frac{a_2}{q^2}+\cdots 
&=\frac{a_1}{q}+\cdots + \frac{a_n}{q^n}+\frac{M}{(q-1)q^n}\\
&\ge x-\frac{1}{q^n}+\frac{M}{(q-1)q^n}\\
&\ge x
\end{align*}
by our assumption $M\ge q-1$.
\end{proof}

\begin{definition}
If $M\ge q-1$ and $0< x\le M/(q-1)$, then $(a_i)$ is called the {\em
quasi-greedy} expansion of $x$. (The terminology will be clarified in the next
section.)
\end{definition}

\begin{remark}
Since it is the lexicographically largest infinite sequence satisfying
\eqref{21}, the quasi-greedy expansion is the lexicographically largest
infinite expansion. 
\end{remark}

One can recognize the quasi-greedy expansions by their form:

\begin{theorem}\label{t22}\mbox{}

(a) The map $q\mapsto (\alpha_i)$, where $(\alpha_i)$ denotes the 
quasi-greedy expansion of $1$, is a strictly increasing one-to-one
correspondence
between the interval $(1,M+1]$ and the set of {\em infinite} sequences
satisfying
\begin{equation}\label{23}
\alpha_{n+1}\alpha_{n+2}\ldots\le \alpha_1\alpha_2\ldots
\quad\text{whenever}\quad
\alpha_n<M.
\end{equation}

(b) Fix $q\in (1,M+1]$ arbitrarily and denote by $(\alpha_i)$ the quasi-greedy 
expansion of $1$. The map $x\mapsto (a_i)$, where $(a_i)$
 denotes the quasi-greedy expansion of $x$,  is a strictly increasing
one-to-one correspondence between the interval $(0,M/(q-1)]$ and the set of
{\em infinite} sequences satisfying
\begin{equation}\label{24}
a_{n+1}a_{n+2}\ldots\le \alpha_1\alpha_2\ldots\quad\text{whenever}\quad a_n<M.
\end{equation}
\end{theorem}

For the proof we need the following result:

\begin{lemma}\label{l23}
Let $(\alpha_i)$ be an arbitrary expansion of $1$. If a sequence $(a_i)$
satisfies the condition
\begin{equation*}
a_{n+1}a_{n+2}\ldots\le \alpha_1\alpha_2\ldots\quad\text{whenever}\quad a_n<M,
\end{equation*}
then we also have
\begin{equation*}
\frac{a_{n+1}}{q^{n+1}}+\frac{a_{n+2}}{q^{n+2}}+\cdots \le \frac{1}{q^n}
\end{equation*}
whenever $a_n<M$.

Consequently, if the sequence $(a_i)$ is also infinite, it is the quasi-greedy
expansion of 
\begin{equation*}
x:=\frac{a_1}{q} + \frac{a_2}{q^2}+\cdots .
\end{equation*}
\end{lemma}

\begin{proof}
Starting with $k_0:=n$ let us define by recurrence a sequence of indices
$k_0<k_1<\cdots$ satisfying for $j=1,2,\ldots$ the conditions
\begin{equation*}
a_{k_{j-1}+i}=\alpha_i\quad\text{for}\quad i=1,\ldots, k_j-k_{j-1}-1,
\quad\text{and}\quad a_{k_j}<\alpha_{k_j-k_{j-1}}. 
\end{equation*}
If we obtain an infinite sequence, then we have
\begin{align*}
\sum _{i=n+1}^\infty \frac{a_i}{q^i}
&= \sum _{j=1}^\infty\sum_{i=1}^{k_j-k_{j-1}}
\frac{a_{k_{j-1}+i}}{q^{k_{j-1}+i}}\\
&\le \sum _{j=1}^\infty \Bigl(\Bigl(\sum_{i=1}^{k_j-k_{j-1}}
\frac{\alpha_i}{q^{k_{j-1}+i}}\Bigr)-\frac{1}{q^{k_j}}\Bigr)\\
&\le\sum _{j=1}^\infty \Bigl(\frac{1}{q^{k_{j-1}}}-\frac{1}{q^{k_j}}\Bigr)\\
&=\frac{1}{q^n}.
\end{align*}

Otherwise we have $(a_{k_N+i})=(\alpha_i)$ after a finite number of steps (we do
not exclude the possibility that $N=0$), and we may conclude as follows:
\begin{align*}
\sum _{i=n+1}^\infty \frac{a_i}{q^i}
&=\Bigl(\sum _{j=1}^N\sum_{i=1}^{k_j-k_{j-1}}\frac{a_i}{q^{k_{j-1}+i}}\Bigr)
+\sum _{i=1}^\infty \frac{a_{k_N+i}}{q^{k_N+i}}\\
&\le \sum _{j=1}^N \Bigl(\Bigl(
\sum
_{i=1}^{k_j-k_{j-1}}\frac{\alpha_i}{q^{k_{j-1}+i}}\Bigr)-\frac{1}{q^{k_j}}\Bigr)
+\sum_{i=1}^\infty\frac{\alpha_i}{q^{k_N+i}}\\
&\le\sum _{j=1}^N \Bigl(\frac{1}{q^{k_{j-1}}}-\frac{1}{q^{k_j}}\Bigr)
+\frac{1}{q^{k_N}}\\
&=\frac{1}{q^n}.\qedhere
\end{align*}

If the sequence $(a_i)$ is infinite, then setting
\begin{equation*}
x:=\frac{a_1}{q} + \frac{a_2}{q^2}+\cdots 
\end{equation*}
our result can be written in the form
\begin{equation*}
\frac{a_1}{q}+\cdots + \frac{a_n}{q^n}\ge x-\frac{1}{q^n}
\end{equation*}
whenever $a_n<M$, and this proves that $(a_i)$ is the quasi-greedy
expansion of $x$.
\end{proof}

\begin{proof}[\rm\bf Proof of Theorem \ref{t22}] The strict increasingness of
both maps follows from the definition of quasi-greedy expansions. 

In order to prove that every quasi-greedy expansion satisfies \eqref{24} it suffices to observe that if $a_n<M$ for some
$n$, then we infer from the inequalities
\begin{equation*}
\frac{a_1}{q}+\cdots+\frac{a_{n-1}}{q^{n-1}}+\frac{a_n+1}{q^n}\ge x
=\frac{a_1}{q}+\frac{a_2}{q^2}+\cdots
\end{equation*}
that 
\begin{equation*}
\frac{a_{n+1}}{q}+\frac{a_{n+2}}{q^2}+\cdots \le 1.
\end{equation*}
This yields \eqref{24} because $(\alpha_i)$ is by definition the
lexicographically largest infinite sequence satisfying such an inequality (see
\eqref{21}. The condition \eqref{23} hence follows by taking $x=1$.

Finally, the {\em onto} property of both maps follows from the preceding lemma.
\end{proof}

\begin{remarks}\mbox{}

\begin{itemize}
\item The results of this section extend to the case $M=\infty$, i.e., when all
nonnegative digits are permitted in the expansions. In this case the conditions
$M\ge q-1$, $x\le M/(q-1)$, $a_n<M$ and $\alpha_n<M$ are automatically
satisfied, and hence can be omitted.

\item Observe that if $0<x\le (M+1)/q$, then the quasi-greedy expansion of $x$ remains
the same by changing $M$ to infinity. Consequently, in Theorem \ref{t22} the
condition \eqref{23} is satisfied for {\em all} $n$, and \eqref{24} is
also satisfied for {\em all} $n$ if $0<x\le (M+1)/q$.
\end{itemize}

\end{remarks}

For the sake of convenience we end this section by giving explicitly the results corresponding to the case $M=\infty$. We fix a real number $q>1$.

\begin{definition}
The \emph{quasi-greedy expansion} of a positive real number $x$ is by definition the lexicographically largest infinite sequence $(a_i)$ of nonnegative integers satisfying
\begin{equation*}
\frac{a_1}{q}+ \frac{a_2}{q^2}+\cdots \le x.
\end{equation*}
\end{definition}

Equivalently, the sequence $(a_i)$ is defined recursively as follows: if $a_k$ has already
been defined for all $k<n$ (no assumption if $n=1$), then $a_n$ is the
largest integer satisfying the inequality
\begin{equation*}
\frac{a_1}{q}+\cdots + \frac{a_n}{q^n}< x.
\end{equation*}

We have the following

\begin{proposition}\label{p24}
If $x>0$, then $(a_i)$ is an infinite expansion of $x$, i.e., it has infinitely many nonzero elements and
\begin{equation*}
\frac{a_1}{q}+ \frac{a_2}{q^2}+\cdots = x.
\end{equation*}

Furthermore, we have $a_n< q$ for all $n\ge 2$.
\end{proposition}

\begin{proof}
We have
\begin{equation*}
\frac{a_1}{q}+\cdots + \frac{a_n}{q^n}< x\le \frac{a_1}{q}+\cdots + \frac{a_n}{q^n}+\frac{1}{q^n}
\end{equation*}
for all $n\ge 1$ by definition. Hence
\begin{equation*}
0<x-\left( \frac{a_1}{q}+\cdots + \frac{a_n}{q^n}\right) \le \frac{1}{q^n}
\end{equation*}
for all $n$ and the right-hand side converges to zero.

It follows from the inequalities 
\begin{equation*}
\frac{a_1}{q}+\cdots + \frac{a_n}{q^n}+\frac{a_{n+1}}{q^{n+1}}< x\le \frac{a_1}{q}+\cdots + \frac{a_n}{q^n}+\frac{1}{q^n}
\end{equation*}
that
\begin{equation*}
\frac{a_{n+1}}{q^{n+1}}< \frac{1}{q^n}
\end{equation*}
and therefore $a_{n+1}<q$ for $n=1,2,\ldots .$
\end{proof}

\begin{theorem}\label{t25}\mbox{}

(a) The map $q\mapsto (\alpha_i)$, where $(\alpha_i)$ denotes the 
quasi-greedy expansion of $1$, is a strictly increasing one-to-one
correspondence
between the interval $(1,\infty)$ and the set of {\em infinite} sequences
satisfying
\begin{equation*}
\alpha_{n+1}\alpha_{n+2}\ldots\le \alpha_1\alpha_2\ldots
\quad\text{for all}\quad
n.
\end{equation*}

(b) Fix $q>1$ arbitrarily and denote by $(\alpha_i)$ the quasi-greedy 
expansion of $1$. The map $x\mapsto (a_i)$, where $(a_i)$
 denotes the quasi-greedy expansion of $x$,  is a strictly increasing
one-to-one correspondence between the interval $(0,\infty)$ and the set of
{\em infinite} sequences satisfying
\begin{equation*}
a_{n+1}a_{n+2}\ldots\le \alpha_1\alpha_2\ldots\quad\text{for all}\quad n.
\end{equation*}
\end{theorem}

\begin{remark}
It follows from Propositions \ref{p21} and \ref{p24} that the quasi-greedy expansion of $x$ is the same for all $M$ satisfying $0<x\le (M+1)/q$ (including $M=\infty$). In particular, the quasi-greedy expansion $(\alpha_i)$ of $x=1$ is the same for all $q-1\le M\le\infty$.
Consequently, in Theorem \ref{t22} the
condition \eqref{23} is satisfied for {\em all} $n$, and \eqref{24} is
also satisfied for {\em all} $n$ if $0<x\le (M+1)/q$.
\end{remark}

\begin{proof}
We may repeat the proof of Lemma \ref{l23} (by simply omitting the words ``whenever $a_n<M$'' in its statement) and Theorem \ref{t22}.
\end{proof}

\section{Greedy expansions}\label{s3}

Fix a positive integer $M$ and a real number $q>1$. We recall that $x$ must belong to the interval $[0,M/(q-1)]$ in order to have an expansion. 

In order to prove a converse statement, let us introduce for each $x\ge 0$ the
lexicographically largest sequence $(b_i)$ satisfying
\begin{equation}\label{31}
\frac{b_1}{q}+ \frac{b_2}{q^2}+\cdots \le x.
\end{equation}
This is equivalent to the following recursive definition: if $b_k$ has already
been defined for all $k<n$ (no assumption if $n=1$), then let $b_n$ be the
largest integer satisfying the inequalities
\begin{equation*}
b_n\le M
\quad\text{and}\quad
\frac{b_1}{q}+\cdots + \frac{b_n}{q^n}\le x.
\end{equation*}
Since $x\ge 0$, the definition is correct. First we prove the following variant of Proposition \ref{p21}:

\begin{proposition}\label{p31}
If $M\ge q-1$ and $0\le x\le M/(q-1)$, then $(b_i)$ is an expansion of $x$.
\end{proposition}

\begin{proof}
The case $x=0$ is obvious: then $(b_i)$ is the null sequence. 

If $x>0$, then comparing with the
recursive  definition of the quasi-greedy expansions we obtain that
\begin{equation*}
\frac{a_1}{q}+ \frac{a_2}{q^2}+\cdots \le 
\frac{b_1}{q}+ \frac{b_2}{q^2}+\cdots \le x
\end{equation*}
for all $n$. If $M\ge q-1$ and $x\le M/(q-1)$, then the left-hand side tends to
$x$ by Proposition \ref{p21} and therefore $(b_i)$ is also an expansion of
$x$. This is obviously satisfied for $x=0$, too, when $(b_i)$ is the null
sequence. 
\end{proof}

\begin{definition}
If $M\ge q-1$ and $0\le x\le M/(q-1)$, then $(b_i)$ is called the {\em
greedy} expansion of $x$. 
\end{definition}

\begin{remark}
As the lexicographically largest sequence satisfying \eqref{31}, the greedy
expansion of $x$ is the lexicographically largest expansion of $x$. 
\end{remark}

Next we prove a variant of Theorem \ref{t22}. It is convenient to
define the greedy expansion of $1$ for $q=1$ by setting
$(\beta_i)=10^{\infty}$, i.e., $\beta_1=1$ and $\beta_i=0$ for all $i>1$.

\begin{theorem}\label{t32}\mbox{}

(a) The map $q\mapsto (\beta_i)$, where $(\beta_i)$ denotes the 
greedy expansion of $1$, is a strictly increasing one-to-one
correspondence
between the interval $[1,M+1]$ and the set of all sequences
satisfying
\begin{equation}\label{32}
\beta_{n+1}\beta_{n+2}\ldots< \beta_1\beta_2\ldots
\quad\text{whenever}\quad
\beta_n<M.
\end{equation}

(b) Fix $1<q\le M+1$ arbitrarily and denote by $(\alpha_i)$ the
{\em quasi-greedy} expansion of $1$. The map $x\mapsto (b_i)$, where $(b_i)$
 denotes the greedy expansion of $x$,  is a strictly increasing
one-to-one correspondence between the interval $[0,M/(q-1)]$ and the set of
all sequences satisfying
\begin{equation}\label{33}
b_{n+1}b_{n+2}\ldots< \alpha_1\alpha_2\ldots\quad\text{whenever}\quad b_n<M.
\end{equation}
\end{theorem}

\begin{remark}
Part (b) of this theorem is a slight generalization of earlier theorems obtained by Parry \cite{Par1960} and Dar\'oczy and K\'atai \cite{DarKat1995}.
\end{remark}

We need a variant of Lemma \ref{l23}:

\begin{lemma}\label{l33}
If an expansion $(\beta_i)$ of $1$ for some $q\ge 1$ satisfies the condition
\begin{equation*}
\beta_{n+1}\beta_{n+2}\ldots< \beta_1\beta_2\ldots\quad\text{whenever}\quad
\beta_n<M,
\end{equation*}
then it is the greedy expansion of $1$.
\end{lemma}

\begin{proof}
The case $q=1$ is obvious, so we assume henceforth that $q>1$.
Defining the sequence $k_0<k_1<\cdots$ as in the proof of Lemma
\ref{l23}, now we always obtain an infinite sequence, and hence
\begin{align*}
\sum _{i=n+1}^\infty \frac{\beta_i}{q^i}
&= \sum _{j=1}^\infty\sum_{i=1}^{k_j-k_{j-1}}
\frac{\beta_{k_{j-1}+i}}{q^{k_{j-1}+i}}\\
&\le \sum _{j=1}^\infty \Bigl(\Bigl(\sum_{i=1}^{k_j-k_{j-1}}
\frac{\beta_i}{q^{k_{j-1}+i}}\Bigr)-\frac{1}{q^{k_j}}\Bigr)\\
&\le\sum _{j=1}^\infty \Bigl(\frac{1}{q^{k_{j-1}}}-\frac{1}{q^{k_j}}\Bigr)\\
&=\frac{1}{q^n}
\end{align*}
whenever $\beta_n<M$. It remains to exclude the equality here. In the
last computation we have equality only if $(\beta_i)$ has a last nonzero term
$\beta_m$ and if the sequence $(\beta_i)$ is periodic with period
$\beta_1\ldots \beta_{m-1}\beta_m^-$, $\beta_m^-=\beta_m-1$. Since then
$(\beta_i)$ is a finite sequence, this implies $m=1$ and $\beta_m=1$.
However,this corresponds to the case $q=1$, already excluded.
\end{proof}

\begin{proof}[Proof of the theorem] The strict increasingness of both maps follows from the
definition of greedy expansions. They are onto by Lemmas \ref{l23} and
\ref{l33}. It remains to show that every greedy expansion satisfies the inequalities \eqref{32} and \eqref{33}, respectively.

For the proof of \eqref{32} it suffices to
observe that if $\beta_n<M$ for some
$n$, then we infer from the inequalities
\begin{equation*}
\frac{\beta_1}{q}+\cdots+\frac{\beta_{n-1}}{q^{n-1}}+\frac{\beta_n+1}{q^n}> 1
=\frac{\beta_1}{q}+\frac{\beta_2}{q^2}+\cdots
\end{equation*}
that 
\begin{equation*}
\frac{\beta_{n+1}}{q}+\frac{\beta_{n+2}}{q^2}+\cdots < 1.
\end{equation*}
This shows first that $(\beta_{n+i})\ne (\beta_i)$ because for the latter
sequence we have equality, and secondly that $(\beta_{n+i})\le (\beta_i)$
because $(\beta_i)$ is by definition
the lexicographically largest sequence satisfying 
\begin{equation*}
\frac{\beta_1}{q}+ \frac{\beta_2}{q^2}+\cdots \le 1.
\end{equation*}
Therefore $(\beta_{n+i})< (\beta_i)$.

The beginning of the proof of \eqref{33} is similar: if $b_n<M$ for some
$n$, then we infer from the inequalities
\begin{equation*}
\frac{b_1}{q}+\cdots+\frac{b_{n-1}}{q^{n-1}}+\frac{b_n+1}{q^n}> x
=\frac{b_1}{q}+\frac{b_2}{q^2}+\cdots
\end{equation*}
that 
\begin{equation*}
\frac{b_{n+1}}{q}+\frac{b_{n+2}}{q^2}+\cdots < 1.
\end{equation*}
This shows first that $(b_{n+i})\ne (\alpha_i)$ because for the latter sequence
we have equality. Next, if the sequence $(b_i)$ is infinite, then it also shows
that $(b_{n+i})\le (\alpha_i)$ because $(\alpha_i)$ is by definition
the lexicographically largest infinite sequence satisfying 
\begin{equation*}
\frac{\alpha_1}{q}+ \frac{\alpha_2}{q^2}+\cdots \le 1.
\end{equation*}
Therefore $(b_{n+i})< (\alpha_i)$.

For $x=0$ this conditon is obviously satisfied, too. In case $(b_i)$ has a last
nonzero digit $b_m$ we may apply the above argument to the infinite expansion
of $x$ with period $b_1\ldots b_{m-1}b_m^-$ where $b_m^-=b_m-1$, to obtain 
\begin{equation}\label{34}
(b_1\ldots b_{m-1}b_m^-)^{\infty}<(\alpha_i)
\end{equation}
with obvious notation. If \eqref{33} were not satisfied, then we would infer
from the relations
\begin{equation*}
b_1\ldots b_{m-1}b_m^- \le \alpha_1\ldots\alpha_m <b_1\ldots b_{m-1}b_m
\end{equation*}
that 
\begin{equation*}
\alpha_1\ldots\alpha_{m-1}\alpha_m=b_1\ldots b_{m-1}b_m^-
\end{equation*}
and therefore \eqref{34} may be rewritten as
\begin{equation*}
(\alpha_1\ldots\alpha_{m-1}\alpha_m)^{\infty}<(\alpha_i).
\end{equation*}
However this is impossible because $(\alpha_{n+i})\le (\alpha_i)$ for every $n$
by condition \eqref{23} of Theorem \ref{t22} and by the remark at the end of
the preceding section, and hence 
\begin{equation*}
(\alpha_i)\le (\alpha_1\ldots\alpha_{m-1}\alpha_m)^{\infty}.\qedhere
\end{equation*}
\end{proof}

Next we clarify the relations between quasi-greedy and greedy
expansions:

\begin{proposition}\label{p34}
Given $1<q\le M+1$ and $0<x\le M/(q-1)$, let us denote by $(a_i)$, $(b_i)$ the
quasi-greedy and greedy expansions of $x$, and by $(\alpha_i)$, $(\beta_i)$
those of $1$. \smallskip

(a) If $(b_i)$ has a last nonzero element $b_m$, then 
\begin{equation*}
(a_i)=b_1\ldots b_{m-1}b_m^-\alpha_1\alpha_2\ldots
\end{equation*}
with $b_m^-=b_m-1$.
Otherwise we have $(a_i)=(b_i)$.
\smallskip

(b) If $(\beta_i)$ has a last nonzero element $\beta_m$, then $(\alpha_i)$ is
periodic with the smallest period $\beta_1\ldots\beta_{m-1}\beta_m^-$ where
$\beta_m^-=\beta_m-1$, i.e., 
\begin{equation*}
(\alpha_i)=(\beta_1\ldots\beta_{m-1}\beta_m^-)^\infty.
\end{equation*}
Otherwise we have $(\beta_i)=(\alpha_i)$, and this sequence is
periodic only in the extreme case $q=M+1$ when
$(\beta_i)=(\alpha_i)=M^{\infty}$.
\end{proposition}

\begin{proof}\mbox{}

(a) The only nontrivial property is that if $(b_i)$ has a last nonzero element
$b_m$, then $(c_i)\le b_1\ldots b_{m-1}b_m^-\alpha_1\alpha_2\ldots$ for every
infinite expansion of $x$. Since $(c_i)<(b_i)$, we must have $c_1\ldots 
c_m\le b_1\ldots b_{m-1}b_m^-$. If we have equality here, then
$c_{m+1}c_{m+2}\ldots$ is an infinite expansion of $1$, so that
$c_{m+1}c_{m+2}\ldots\le
\alpha_1\alpha_2\ldots$.
\smallskip

(b) If $(\beta_i)$ has a last nonzero element $\beta_m$, then
$(\beta_1\ldots\beta_{m-1}\beta_m^-)^\infty$ is clearly an infinite expansion
of $1$. (Observe that the case $m=1$ and $\beta_1=1$ is excluded by our
assumption $q>1$.) If $(\gam_i)$ is an infinite expansion of $1$, then
$\gam_1\ldots\gam_m<\beta_1\ldots\beta_m$, so that
$\gam_1\ldots\gam_m\le\beta_1\ldots\beta_{m-1}b_m^-$. In case of equality
$\gam_{m+1}\gam_{m+2}\ldots$ is again an infinite expansion of $1$, so that 
$\gam_{m+1}\gam_{m+2}\ldots\le\beta_1\ldots\beta_{m-1}b_m^-$. Iterating this
reasoning we find that $(\gam_i)\le (\beta_1\ldots\beta_{m-1}\beta_m^-)^\infty$.
This proves that $(\beta_1\ldots\beta_{m-1}\beta_m^-)^\infty$ is the largest
infinite expansion of  $1$.

The latter sequence cannot have a shorter period of length $k<m$ because then
$k$ should divide $m$ and condition \eqref{32} would be violated for $n=m-k$:
putting $\beta_k^+=\beta_k+1$ we would have
\begin{equation*}
\beta_{m-k+1}\ldots\beta_m=\beta_1\ldots\beta_{k-1}\beta_k^+>\beta_1\ldots\beta_
k.
\end{equation*}

Finally, if $(\beta_i)=(\alpha_i)$ has a smallest period
$\beta_1\ldots\beta_m$, then $\beta_m=M$, for otherwise
$\beta_1\ldots\beta_{m-1}\beta_m^+0^{\infty}$ would be a larger expansion of
$1$. Then applying \eqref{23} for $(\beta_i)=(\alpha_i)$ we have 
\begin{equation*}
\alpha_m\alpha_{m+1}\ldots\alpha_{2m-1}\le \alpha_1\alpha_2\ldots\alpha_m,
\end{equation*}
i.e.,
\begin{equation*}
M\alpha_1\ldots\alpha_{m-1}\le \alpha_1\alpha_2\ldots\alpha_m,
\end{equation*}
This yields successively $\alpha_1=M$, $\alpha_2=M$, \ldots, $\alpha_{m-1}=M$,
so that $(\beta_i)=(\alpha_i)=M^{\infty}$.
\end{proof}

\begin{remarks}\mbox{}

\begin{itemize}
\item The proposition and its proof allow us to determine all expansions of $x$
between $(a_i)$ and $(b_i)$:

\begin{itemize}
\item If $(b_i)$ is infinite, then $(a_i)=(b_i)$.

\item If $(b_i)$ has a last nonzero element $b_m$ and if $(\beta_i)$ is
infinite, then $(a_i)=b_1\ldots b_{m-1}b_m^-\beta_1\beta_2\ldots$ is the second
largest expansion of $x$.

\item If both $(b_i)$ and $(\beta_i)$ have last nonzero element $b_m$ and
$\beta_n$, then the expansions $(c_i)$ of $x$ satisfying $(a_i)<(c_i)<(b_i)$
are given by the decreasing sequence
\begin{equation*}
b_1\ldots b_{m-1}b_m^-(\beta_1\ldots \beta_{n-1}\beta_n^-)^N(\beta_1\ldots
\beta_n)0^{\infty},\quad N=0,1,\ldots .
\end{equation*}
\end{itemize}

\item The results of this section extend to $M=\infty$, too. Then the
conditions $M\ge q-1$, $x\le M/(q-1)$, $b_n<M$ and $\beta_n<M$ may be omitted. The greedy expansions in case $M=\infty$ are exactly the beta-expansions introduced by R\'enyi \cite{Ren1957}.

\item If $0\le x<(M+1)/q$, then he greedy expansion of $x$ does not change by changing $M$ to infinity. Therefore condition
\eqref{32} of Theorem \ref{t32} is satisfied for {\em all} $n$, and condition
\eqref{33} is satisfied for {\em all} $n$ if $0\le x\le (M+1)/q$ (the obvious
cases $M= q-1$ and $x=(M+1)/q$ can be checked separately).
\end{itemize}
\end{remarks}

We end this paper by giving again the explicit statements concerning the case $M=\infty$.

\begin{definition}
The \emph{greedy expansion} of a nonnegative real number $x$ is by definition the lexicographically largest sequence $(b_i)$ of nonnegative integers satisfying
\begin{equation*}
\frac{b_1}{q}+ \frac{b_2}{q^2}+\cdots \le x.
\end{equation*}
\end{definition}

Equivalently, the sequence $(b_i)$ is defined recursively as follows: if $b_k$ has already
been defined for all $k<n$ (no assumption if $n=1$), then $b_n$ is the
largest integer satisfying the inequality
\begin{equation*}
\frac{b_1}{q}+\cdots + \frac{b_n}{q^n}\le x.
\end{equation*}

We have the following

\begin{proposition}\label{p35}
If $x\ge 0$, then $(b_i)$ is an expansion of $x$, i.e.,
\begin{equation*}
\frac{b_1}{q}+ \frac{b_2}{q^2}+\cdots = x.
\end{equation*}

Furthermore, we have $b_n< q$ for all $n\ge 2$.
\end{proposition}

\begin{proof}
We have
\begin{equation*}
\frac{b_1}{q}+\cdots + \frac{b_n}{q^n}\le x< \frac{b_1}{q}+\cdots + \frac{b_n}{q^n}+\frac{1}{q^n}
\end{equation*}
for all $n\ge 1$ by definition. Hence
\begin{equation*}
0\le x-\left( \frac{b_1}{q}+\cdots + \frac{b_n}{q^n}\right) < \frac{1}{q^n}
\end{equation*}
for all $n$ and the right-hand side converges to zero.

It follows from the inequalities 
\begin{equation*}
\frac{b_1}{q}+\cdots + \frac{b_n}{q^n}+\frac{b_{n+1}}{q^{n+1}}\le x< \frac{b_1}{q}+\cdots + \frac{b_n}{q^n}+\frac{1}{q^n}
\end{equation*}
that
\begin{equation*}
\frac{b_{n+1}}{q^{n+1}}< \frac{1}{q^n}
\end{equation*}
and therefore $b_{n+1}<q$ for $n=1,2,\ldots .$
\end{proof}

The following result is essentially due to Parry \cite{Par1960}; following Dar\'oczy and K\'atai \cite{DarKat1995} we simplify its formulation by using quasi-greedy expansions. It is convenient to introduce again the greedy expansion of $1$ for $q=1$ by setting
$(\beta_i)=10^{\infty}$, i.e., $\beta_1=1$ and $\beta_i=0$ for all $i>1$.

\begin{theorem}\label{t36}\mbox{}

(a) The map $q\mapsto (\beta_i)$, where $(\beta_i)$ denotes the 
greedy expansion of $1$, is a strictly increasing one-to-one
correspondence
between the interval $[1,\infty)$ and the set of all sequences
satisfying
\begin{equation*}
\beta_{n+1}\beta_{n+2}\ldots< \beta_1\beta_2\ldots
\quad\text{for all}\quad n.
\end{equation*}

(b) Fix $q>1$ arbitrarily and denote by $(\alpha_i)$ the
{\em quasi-greedy} expansion of $1$. The map $x\mapsto (b_i)$, where $(b_i)$
 denotes the greedy expansion of $x$,  is a strictly increasing
one-to-one correspondence between the interval $[0,\infty)$ and the set of
all sequences satisfying
\begin{equation*}
b_{n+1}b_{n+2}\ldots< \alpha_1\alpha_2\ldots
\quad\text{for all}\quad n.
\end{equation*}
\end{theorem}

\begin{proof}
We may repeat the proof of Lemma \ref{l33} (by changing the words ``whenever $\beta_n<M$'' to ``for all $n\ge 1$'') and Theorem \ref{t32} (by changing the words ``if $\beta_n<M$'' to ``for all $n\ge 1$'').
\end{proof}

Next we clarify the relations between quasi-greedy and greedy
expansions:

\begin{proposition}\label{p37}
Given $q>1$ and $x>0$, let us denote by $(a_i)$, $(b_i)$ the
quasi-greedy and greedy expansions of $x$, and by $(\alpha_i)$, $(\beta_i)$
those of $1$. \smallskip

(a) If $(b_i)$ has a last nonzero element $b_m$, then 
\begin{equation*}
(a_i)=b_1\ldots b_{m-1}b_m^-\alpha_1\alpha_2\ldots
\end{equation*}
with $b_m^-=b_m-1$.
Otherwise we have $(a_i)=(b_i)$.
\smallskip

(b) If $(\beta_i)$ has a last nonzero element $\beta_m$, then $(\alpha_i)$ is
periodic with the smallest period $\beta_1\ldots\beta_{m-1}\beta_m^-$ where
$\beta_m^-=\beta_m-1$, i.e., 
\begin{equation*}
(\alpha_i)=(\beta_1\ldots\beta_{m-1}\beta_m^-)^\infty.
\end{equation*}
Otherwise we have $(\beta_i)=(\alpha_i)$, and this sequence is not
periodic.
\end{proposition}

\begin{proof}
We may repeat the proof of Proposition \ref{p34}, by shortening its last paragraph by observing that if $(\beta_i)=(\alpha_i)$ had a period of length $m$, then $\beta_1\ldots\beta_{m-1}\beta_m^+0^{\infty}$ would also be an expansion of $1$, contradicting the maximality of $(\beta_i)$.
\end{proof}

\end{document}